\documentclass{amsart}
\usepackage{amssymb,mathrsfs}%,skak}
\usepackage{color}

\usepackage[textheight=25 cm, textwidth=16cm, headheight=10pt, headsep=20pt, footskip=10pt, left=1.5cm, right=1.5cm]{geometry}
\def\bN {\mathbf{N}}

\def\bR {\mathbf{R}}

\def\fH {\mathfrak{H}}

\def\fS {\mathfrak{S}}

\def\\fH {\mathcal{H}}

\def\b {{\beta}}

\def\\hbar {{\\hbarilon}}

\def\si {{\sigma}}

\newcommand{\Lip}{\operatorname{Lip}}

\newcommand{\ba}{\begin{aligned}}
\newcommand{\ea}{\end{aligned}}

\newcommand{\be}{\begin{equation}}
\newcommand{\ee}{\end{equation}}

\newcommand{\bea}{\begin{eqnarray}}
\newcommand{\eea}{\end{eqnarray}}

\newcommand{\lb}{\label}
\newtheorem{Thm}{Theorem}[section]
\newtheorem{Rmk}[Thm]{Remark}

\newtheorem{Cor}[Thm]{Corollary}
\newtheorem{Lem}[Thm]{Lemma}
\newtheorem{Def}[Thm]{Definition}

\renewcommand{\hbar}{{}}

\def\bN {\mathbf{N}}

\def\bR {\mathbf{R}}

\newcommand{\gen}{\gamma}
\usepackage{hyperref}
\hypersetup{pdfborder=0 0 0}
\begin{document}

\title{\Large On The Mean Field limit for  Cucker-Smale models}

%\author{R.N., T.P.}
\author[R. Natalini]{Roberto Natalini}
\address[R.N]{IAC, Via dei Taurini, 19, 00185 Roma RM Italy}
\email{roberto.natalini@cnr.it}
\author[T. Paul]{Thierry Paul}
\address[T.P.]{CNRS \& LJLL Sorbonne Université 4 place Jussieu 75005 Paris, France France}
\email{thierry.paul@upmc.fr}

\begin{abstract}
\Large 
In this   note, we consider generalizations of the Cucker-Smale dynamical system and we derive rigorously in Wasserstein's type topologies the mean-field limit (and propagation of chaos) to the Vlasov-type equation introduced in \cite{ht}. 
Unlike previous results on the Cucker-Smale model, our approach is not based on the empirical measures, 
but, using an Eulerian point of view introduced in \cite{gmp} in the Hamiltonian setting, 
we show the limit providing explicit constants.
%Using an Eulerian point of view introduced in \cite{gmp} in the Hamiltonian setting, we  don't use empirical measures and provide explicit constants. 
Moreover, for non strictly Cucker-Smale particles dynamics, we also give an insight on what induces a flocking behavior of the solution to the Vlasov equation to the - unknown a priori -  flocking properties of the original particle system.
%The vector field we consider   is sublinear Lipschitz continuous and the initial data are compactly supported.
\end{abstract}
\LARGE
\maketitle

\tableofcontents

\section{Introduction }\label{intro}
   The Cucker-Smale  model \cite{CS1,CS2} is a  particles system which exhibits the so-called flocking phenomenon, namely the dynamical property of alignment of all the velocities  and gathering of all the positions asymptotically when the time of evolution   diverges. Its study has developed an intense activity these last years, see for example the intensive bibliography in the recent paper \cite{hakimzhang}.

It consists in the following vector field on $\bR^{2dN}$ 
\be\label{CS}
\left\{\begin{array}{rcl}
\dot {x_i}&=&v_i\\
\dot {v_i}&=&
\frac1N\sum\limits_{j=1}^N\psi(x_i-x_j)(v_j-v_i).
\end{array}
\right.
\hskip 1cm x_i,v_i\in\bR^d,\ 
\ i=1,\dots,N
\ee
where $\psi:\bR^d\to\bR$ is a bounded Lipschitz continuous positive nonincreasing function.

Associated to \eqref{CS},  the following Vlasov type kinetic equation on $\bR^{2d}$ was introduced in \cite{ht}:
\be\label{vlasovcs}
\partial_t\rho^t(x,v)+v\cdot\nabla_x\rho^t(x,v)+\nabla_v\cdot\left(\rho^t(x,v)\int_{\bR^{2d}}\psi(x-y)(v-w)\rho^t(y,w)dydw \right)=0.
\ee
So far the kinetic non-linear equation \eqref{vlasovcs} has been derived in a Lagrangian point of view, i.e. by following the trajectories solving the vector field \eqref{CS} in the so-called empirical measure
\be\label{[defempmes}
\Pi_{(x_1,\dots,x_n;v_1,\dots,v_N)}(x,v):=\frac1N\sum_{i=1}^N
\delta(x-x_i)\delta(v-v_i).
\ee
Indeed, let $\Phi^t$  be the flow  generated by the system \eqref{CS} and let us define
\be\label{defempmest}
\Pi^t_{(x_1,\dots,x_n;v_1,\dots,v_N)}:=\Pi_{\Phi^{-t}(x_1,\dots,x_n;v_1,\dots,v_N)}.
\ee
 One shows easily that $\Pi^t_{(x_1,\dots,x_n;v_1,\dots,v_N)}(x,v)$ solves \eqref{vlasovcs} for each $N$.

This point of view follows directly the Dobrushin way \cite{Dobrush}
 of deriving the Vlasov equation for the large $N$ limit of Hamiltonian vector fields. The Cucker-Smale  model \eqref{CS} and  its Vlasov type associated equation \eqref{vlasovcs} have been extensively studied in \cite{hakimzhang} and we quote the following of their results:  the function $\psi$ is supposed to be positive,  bounded,  uniformly Lipschitz continuous and  
 %satisfying $(\psi(r_1)-\psi(r_2))(r_1-r_2)~\leq~ 0, r_1,r_2>0$
  nonincreasing.
\begin{enumerate}
\item\label{item1}
%there exists two positive constants $C,D$, independent of $N$ such that 
the solutions of \eqref{CS} satisfy for all $t$

$$
\sup_{i,j\leq N}\|x_i(t)-x_j(t)\|\leq C \ \mbox{ and } \ 
\sup_{i,j\leq N}\|v_i(t)-v_j(t)\|\leq e^{-Dt}
$$

where $C$ and $D$ are two positive constants depending only on 

$\sup_{i,j\leq N}\|x_i(0)-x_j(0)\|$ and 
$\sup_{i,j\leq N}\|v_i(0)-v_j(0)\|$.

\item\label{item2}
let $\Pi_{(x_1,\dots,x_n;v_1,\dots,v_N)}\to\mu(0)$ as $N\to\infty$, then
$$
\overline{\lim_{N\to\infty}}\sup_{t\in\bR}W_p(\Pi^t_{(x_1,\dots,x_n;v_1,\dots,v_N)}-\mu(t))=0
$$
where $W_p$ is the Wasserstein distance of exponent $p\in\bN$ (see definition in Remark \ref{highorders} below).

\item\label{item3}
Let $\int_{\bR^{2d}}d\mu(0)=1,\ \int_{\bR^{2d}}vd\mu(0)=:\bar v,\ \int_{\bR^{2d}}v^2d\mu(0)<\infty$. Then, for some $E,F>0$ and every $p$,
$$
\left(\int\|v-\bar v\|^pd\mu(t)\right)^{\frac1p}\leq Ee^{-Ft}.
$$

\end{enumerate}
\vskip 1cm

Recently,  a more Eulerian way of deriving the Vlasov equation  associated to  Hamiltonian vector fields was introduced (\cite{gmp}, after \cite{gmr}). It consists in considering the movement of a generic particle solution to the Liouville equation associated to an Hamiltonian system, rather than considering the empirical measure associated to the Hamiltonian flow.
 Transposed in the (non Hamiltonian) present situation, the method leads to the  following.

%\vskip 1cm, and prove
%
%
%
%
%The standard $\si$ is given by ......

We  will consider the pushforward\footnote{We recall that the pushforward of a measure $\mu$ by a measurable function $\Phi$ is $\Phi\#\mu$ defined by $\int \varphi d(\Phi\#\mu):=\int (\varphi\circ f)d\mu$ for every measurable function $f$.} by the flow $\Phi^t$ associated to \eqref{CS} of a compactly supported, symmetric by permutations of the variables $N$-body  probability density 
$\rho^{in}_N$  on phase space $\bR^{2dN}$.

In order to describe the movement of a ``generic" particle 
in the limit of diverging number of particles, 
we want to perform an average on the $N$ particles but one. This means that
we consider the first marginal of 
\be\label{defrhoNt}
\Phi^t\#\rho^{in}_N:=\rho^t_N
\ee 
that is
\be\label{marg}
\rho^t_{N;1}(x,\xi):=\int_{\bR^{2(d-1)N}}
\rho^t_N(x,\xi,x_2,\xi_2,\dots,x_N,\xi_N)dx_2d\xi_2\dots dx_Nd\xi_N.
\ee

Then, when $\rho^{in}_N$ is factorized as $\rho^{in}_N=(\rho^{in})^{\otimes N}$, where $\rho^{in}$ is a probability measure on $\bR^{2d}$, we will prove that the marginal of $\rho^t_N$ tends, as $N\to\infty$,  to    the solution of the  effective non-linear Vlasov type  equation \eqref{vlasovcs} with 
%an initial condition  
$\rho^{t=0}=\rho^{in}\in L^1(\bR^{2d},dxdv)$.

%Such of   Vlasov-type equation associated to \eqref{liouvnd} has be 
%introduced in \cite{ht} and which reads
%\be\label{vlasnd0}
%\partial_t
%\rho^t(x,v)
%+v.\nabla_x\rho^t+
%%G_{\rho_1}.\nabla_v\rho_1
%\nabla_v(G_{\rho^t}\rho^t)=0
%,\ \ \
% \rho^{t=0}=\rho^{in}\in L^1(\bR^{2d},dxdv),
%\ee
%with, for every $L^1$ function $\rho$,
%\be\label{defgrho0}
%G_{\rho}(x,v)=\int_{\bR^{2d}}\psi(x-y)(v-w)\rho(y,w)dydw.
%% +F(x).
%\ee
In fact, we can prove (see Remark \ref{highorders} below) that the marginals at every order $n$ of $\rho_N^t$ as defined by
\be\label{margn}
\rho^t_{N;n}(x_1,\xi_1,\dots,x_n,\xi_n):=\int_{\bR^{2d(N-n)}}
\rho^t_N(x_i,\xi_1,x_2,\xi_2,\dots,x_N,\xi_N)dx_{n+1}d\xi_{n+1}\dots dx_Nd\xi_N,
\ee
tend, in Wasserstein topology of any exponent $p\geq 1$, to the $n$th tensorial power of the solution $\rho^t$ of the Vlasov equation, i.e. $\rho^t_{N;n}\to(\rho^t)^{\otimes n}$. Nevertheless, for sake of clarity of this short note, we will present in detail only the case $n=1$.

This result (for $n=1$), i.e. Theorem \ref{main} below, has to be put in correspondence with the item \ref{item2} of the results of \cite{hakimzhang} just described,  but there are several differences. First, our results hold true for more general initial data than the empirical measures. Second, our result will not be uniform in time as in item \ref{item1}. Third,  we will get an explicit rate of convergence in the asymptotic $N\to\infty$.

Our methods will also apply to systems of the general form
\be\label{genCS}
\left\{\begin{array}{rcl}
\dot {x_i}&=&v_i\\
\dot {v_i}&=&
\frac1N\sum\limits_{j=1}^N\gen(x_i-x_j,v_i-v_j)
\end{array}
\right.
\hskip 1cm x_i,v_i\in\bR^d,\ 
\ i=1,\dots,N,
\ee
where $\gen(x,v):\bR^{2d}\to\bR^d$ is a  Lipschitz continuous function,  bounded in $(x,v)$, Corollary \ref{maingen} or bounded in $x$ and sublinear in $v$, Theorem \ref{mainsublin} and Corollary \ref{maingensublinest}. 

In the following, we will denote by $\Lip(\gen)_{(x,v)}$ the (local) Lipschitz constant of $\gen$ at the point $(x,v)\in\bR^{2d}$ and we will suppose, without loss of generality, that there exists $\gen_0>0$ such that, for all $(x,v)\in\bR^{2d}$,
\begin{eqnarray}
\gen(x,v)&\leq&\gen_0|v|\nonumber\\
\Lip(\gen)_{(x,v)}&\leq&\gen_0|v|\nonumber
\end{eqnarray}
where $|\cdot|$ denotes the Euclidean norm on $\bR^d$.

We will show in Appendix \ref{dynestpar}  the gobal existence  of the solutions to \eqref{genCS}, with an exponential growth in time with respect to the initial conditions. Of course, uniqueness is  a consequence of (the iteration in time of)  the Cauchy-Lipschitz Theorem.

The Cucker-Smale vector fields satisfies this assumptions with $\gen_0=\|\psi\|_{L^\infty}$.

Note that  the following Vlasov type kinetic equation on $\bR^d$ is obviously associated to \eqref{genCS}, in a  natural way:
\be\label{vlasovgencs}
\partial_t\rho^t(x,v)+v\cdot\nabla_x\rho^t(x,v)+\nabla_v\cdot\left(\rho^t(x,v)\int_{\bR^{2d}}\gen(x-y,v-w)\rho^t(y,w)dydw \right)=0.
\ee
Let us immediately notice that our assumptions on $\gen$ contain the two assumptions in Theorem 2.3 in \cite{prt}. The first one because of the Lipschitz and sublinearity properties of $\gen$, and the second one thanks to the fact, see \cite{VillaniAMS,VillaniTOT}, that
$$
\sup_{|\Lip(\varphi)|\leq 1}|\int \varphi(\mu-\nu)|\leq W_1(\mu,\nu)
$$
where $W_1$ is the Wasserstein distance of exponent $1$, as defined in Remark \ref{highorders} in the present paper.
Therefore, thanks to the results of \cite{prt} (see also \cite{carr2}), there exists a unique continuous solution  to \eqref{vlasovgencs} in $ C_0(\bR,\mathcal P_c(\bR^{2d}))$, where 
$\mathcal P_c(\bR^{2d})$ is the space of compactly supported probability measures on $\bR^{2d}$. 

Our last result,  Theorem \ref{inverse} and Remark \ref{margplusinv},  gives an insight of the ``inverse problem" of the mean-field limit in the case where the solution to  the Vlasov kinetic equation \eqref{vlasovgencs}  exhibits a flocking behavior of the form \eqref{suppcs} below, without any knowledge concerning the flocking behavior of the original particles system. Namely,  when the support of the solution to the Vlasov equation remains of finite size in the variable $x$ and reduces asymptotically exponentially in time to a single point in  the momentum variable $v$, we prove an approximate similar behavior for all the marginals of initial distributions %of particles 
pushforwarded by the $N$-body dynamics, for $N$ large enough.
\vskip 0.5cm
Let us finish this introduction by quoting  very  few references among the huge number of works dedicated to the rigorous derivation of the mean-field limit of particles systems, after the pioneering work \cite{Dobrush} already mentioned: \cite{szn, carr1,jw2} for stochastic systems and using empirical measures,   \cite{haliu,hakimzhang} (already mentioned) specifically for the Cucker-Smale model and using empirical measures too, \cite{gmr,gmp} (already mentioned too) for globally Lipschitz forces, \cite{jw1} for rough but bounded forces and finally \cite{gp2} for a derivation using the full hierarchy of equations satisfied by the marginals, for Hamiltonian systems with analytic potentials.
\section{The results}\label{result}
In this section we will state our main results, but let us start  by recalling the definition of the second order Wasserstein distance $W_2$ (see \cite{ambrosiosavare,VillaniAMS,VillaniTOT}).
\begin{Def}[quadratic Wasserstein distance]\label{defwas}
The Wasserstein distance of order two between two probability measures $\mu,\nu$ on $\bR^m$ with finite second moments is defined as
$$
%\MKd
W_2(\mu,\nu)^2
=
\inf_{\gamma\in\Gamma(\mu,\nu)}\int_{\bR^m\times\bR^m}
|x-y|^2\gamma(dx,dy)
$$
where $\Gamma(\mu,\nu)$ is the set of probability measures on $\bR^m\times\bR^m$ whose marginals on the two factors are $\mu$ and $\nu$. 

That is to say that elements $\gamma$ of $\Gamma(\mu,\nu)$, called couplings or transportation maps or transport maps or just maps, depending of the authors, satisfy, for every test functions $a$ and $b$ in $C_0(\bR^d)$,
$$
 \int_{\bR^{2m}}a(x)\gamma(dx,dy)=\int_{\bR^m}a(x)\mu(dx),\ \ \ 
 \int_{\bR^{2m}}b(y)\gamma(dx,dy)=\int_{\bR^m}b(y)\nu(dy).
 $$
\end{Def}
The symmetry property in $\mu,\nu$ is obvious and the separability property is easily proven by taking the optimal coupling between $\mu$ and itself equal to
\be\label{coupop}
\gamma=\mu\delta(x-y).
\ee
Conversely, if $W_2(\mu,\nu)^2=0$, then $
\int_{\bR^m\times\bR^m}
|x-y|^2\gamma(dx,dy)=0$ for some $\gamma$ so that $x=y$ $\gamma$ a.e. and, for every Borel function $\varphi$, 
$$
\int\varphi(x)\mu(dx)=\int\varphi(x)\gamma(dx,dy)=\int\varphi(y)\gamma(dx,dy)=
\int\varphi(y)\nu(dy)\Rightarrow\mu=\nu.
$$
The proof of the triangular inequality is much more  involved, see again \cite{ambrosiosavare,VillaniAMS,VillaniTOT}.

\vskip 1cm
We can now state the main result of this  note, proven in Section \ref{Proofmain}. 
\begin{Thm}[Cucker-Smale model]\label{main}\ 

%Let $\rho(t)$ be the solution of \eqref{liouvnd} with initial condition $\rho^{\mbox{in}}=(\rho_1^{in})^{\otimes N}\in L^1(\bR^{2dN})$,
% %symmetrical by permutation
%$\rho_1^{in}\in L^1(\bR^{2d})$ , and let $\rho(t)_1\in L^1(\bR^{2d})$ be defined by 
%\be\label{defmar}
%%(\rho(t))_1(x,v)=\int_{\bR^{2d(N-1)}}\rho(t)(x,v,x_2,\dots,x_N,v_2,\dots,v_N)dx_2\dots dx_N dv_2\dots dv_N\in L^1(\bR^{2d}).
%\rho(t)_1(x,v)=\int_{\bR^{2d(N-1)}}\rho(t)(x,v,x_2,\dots,x_N,v_2,\dots,v_N)dx_2\dots dx_N dv_2\dots dv_N.
%\ee
%Of course $(\rho^{in})_1=\rho^{in}_1$.

Let $\Phi^t$  be the flow  generated by the system \eqref{CS},  $\psi$  bounded positive nonincreasing Lipschitz continuous, and let  $\rho^t$ be the solution to \eqref{vlasovcs} with an initial condition $\rho^{in}\in L^1(\bR^{2d})$ compactly supported.

 Let moreover $\rho^t_{N;1}$ be the first marginal of $\rho^t_N:=\Phi^t\#({\rho^{in}})^{\otimes N}$, as defined in \eqref{marg}.

Then, for all $N>1,\ t\in \bR$,
\be\nonumber
%W_2((\rho(t))_1,\rho_1(t))
W_2(\rho^t_{N;1},\rho^t)
\leq
4
%\frac{
\|\psi\|^2_\infty
(2|\bar v|+|supp[\rho^{in}]|)
%}{N^{\frac12}}
%\|v\|_{L^\infty(\rho^{in})}}{N^{\frac12}}
\left(\frac{e^{Lt}-1}L\right)^{\frac12}
{N^{-\frac12}}
\ee
with
\be\nonumber
%\bar v=\int v\rho^{in}dxdv\ \mbox{ and }\
 L:=2(1+8\|\psi\|^2_\infty\|v\|_{L^\infty(supp[\rho^{in}])}^2)\ and \ \bar v=\int v\rho^{in}dxdv,
%\|\psi\|_\infty
%\sup_{supp(\rho_1(t))}{|v|}\times 
%\frac{e^{\Lambda t}-1}{\Lambda}
%\ \ \ \ \Lambda:=2(1+2(1+\Lip(\phi))^2),
\ee
where  $\|v\|_{L^\infty(supp[\rho^{in}])}:=\sup\limits_{(x,v)\in supp[\rho^{in}]}|v|$.
%where $\phi:\bR^{2d}\to\bR^{2d}$ is any bounded Lipschitz functions equal to the indentity on the support of $\rho^{in}$.
\end{Thm}
%(by $\|v\|_{L^\infty(supp(\rho^{in})})$ we mean $\sup\limits_{(x,v)\in supp[\rho^{in}]}(|v|$).
\begin{Rmk}\label{margplus}
There exists an equivalent result for higher orders marginals of $\rho(t)$ and  other Monge-Kantorovich distances but we prefer in this  note to concentrate on the case of first marginal and quadratic Wasserstein distance. The proof in the more general situation is very close to the one presented here. See Remark \ref{highorders} for some details.
\end{Rmk}
%\begin{Rmk}[estimate]
%\end{Rmk}
%
%ytfrd

%\vskip 3cm

\vskip 1cm
Theorem \ref{main} is actually a corollary of the following more general result, proven in Section \ref{Proofmainsublin}.

%In the sequel of the paper, we will denote by $\Lip(\gen)_{(x,v)}$ the (local) Lipschitz constant of $\gen$ at the point $(x,v)$.
\begin{Thm}[General Cucker-Smale model with general sublinear force]\label{mainsublin}\ 

%Let $\rho(t)$ be the solution of \eqref{liouvnd} with initial co

Let  $\Phi^t$ the flow defined by the system \eqref{genCS}, $\gen(x,v)$   Lipschitz continuous  bounded in $x$ and sublinear in $v$, and let $\rho^t$ be the solution to \eqref{vlasovgencs}, $\gamma$ bounded on $\bR^{2d}$, with an initial condition $\rho_1^{in}\in L^1(\bR^{2d})$  compactly supported.

 Finally, let $\rho^t_{N;1}$ be the first marginal of $\rho^t_N:=\Phi^t\#({\rho^{in}})^{\otimes N}$, as defined in \eqref{marg}.

Then, for all $N>1, t\in \bR$,
\be\nonumber
%W_2((\rho(t))_1,\rho_1(t))
W_2(\rho^t_{N;1},\rho^t)
\leq
C(t)N^{-\frac12}
\ee
with
\begin{eqnarray}\nonumber
C(t)&:=&
\left(
4
\int_0^t
%%\|\gen\|_{{L^\infty(supp(\rho_1(s))}}^{2}
%\sup_{\substack{k,l=1,dots,N\\(Y,\Xi)\in supp[\rho^t_N]}}|\gen(y_k-y_l,\xi_k-\xi_l)|^2
%e^{\int_s^t\L (u)du}ds
%\int_0^t%\sup_{\substack{k,l=1,dots,N\\(Y,\Xi)\in supp[\rho^t_N]}}|\gen(y_k-y_l,\xi_k-\xi_l)|^2
\sup_{(x,v),(x',v')\in supp[\rho^t]}|\gen(x-x',v-v')|^2
e^{\int_s^tL (u)du}ds
\right)^{\frac12}\,
,\nonumber\\
% \L(t)&:=&2(1+2
% %\sup_{\substack{k,l=1,\dots,d\\(Y,\Xi)\in support[\Phi(t)\#(\rho^{in})^{\otimes N}]}}\Lip(\gen)_{(y_k-y_l,\xi_k-\xi_l)}^2
%\sup_{\substack{(x,v),(y,\xi)\in supp[\rho^t]}}\Lip(\gen)_{(x-y,v-\xi)}^2 
% ).\nonumber
\b
L(u)&:=&
2(1+2
\min{(\sup_{\substack{i,l=1,\dots,N\\
(Y,\Xi)\in supp[\rho^u_N]}}\Lip{(\gen)}^2_{(y_i-y_l,\xi_i-\xi_l)},\sup_{\substack{
(x,v),(x',v')\in supp[\rho^u]}}\Lip{(\gen)}^2_{(x-x',v-v')})})
).\nonumber
\end{eqnarray}
%where $\Phi(t)$ is the flow   generated by \eqref{eq1}.
\end{Thm}
The following result is elementarly derived from Theorem \ref{mainsublin}.
\begin{Cor}[General Cucker-Smale model with bounded global Lipschitz force]\label{maingen}\ 
%Let $\rho(t)$ be the solution of \eqref{liouvnd} with initial co

Let  $\rho^t$ be the solution to \eqref{vlasovgencs}, $\gamma$ globaly Lipschitz and bounded on $\bR^{2d}$, with an initial condition $\rho^{in}\in L^1(\bR^{2d})$  and $\Phi^t$ the flow deinfed by the system \eqref{genCS}. Finally, let $\rho^t_{N;1}$ be the first marginal of $\rho^t_N:=\Phi^t\#({\rho^{in}})^{\otimes N}$, as defined in \eqref{marg}.

Then, for all $N>1, t\in \bR$,
\be\nonumber
%W_2((\rho(t))_1,\rho^t)
W_2(\rho^t_{N;1},\rho^t)
\leq
C(t)N^{-\frac12}
\ee
with
\be\nonumber
C(t):=\left(4\|\gen\|_\infty
%\|\psi\|_\infty
%\sup_{supp(\rho_1(t))}{|v|}\times 
\frac{e^{\Lambda t}-1}{\Lambda}\right)^{\frac12}
\ \ \ \ \Lambda:=2(1+2\Lip(\gen)^2).
\ee
%where $\phi:\bR^{2d}\to\bR^{2d}$ is any bounded Lipschitz functions equal to the indentity on the support of $\rho^{in}$.
\end{Cor}

Note that no need of compacity of the support of $\rho^{in}$ is needed any more in Corollary \ref{maingen}.

Theorem \ref{mainsublin} gives a precise estimate involving, through the expression of the functions $C(t)$ and $L(t)$, the knowledge of the size of the support of the initial data propagated by the particle flow $\Phi(t)$ driven by \eqref{vlasovgencs} and the kinetic flow induced by \eqref{vlasovgencs}. This information might be given by the explicit models, i.e. the function $\gen$, as it is the case for the Cucker-Smale model (see also  the very end of this section).

In the general case, the boundness in space and sublinearity in velocities of $\gen$ allow to control the increasing of the flow $\Phi(t)$, and leads to our next result, whose proof is given in Section \ref{proofcor2} below.

\begin{Cor}\label{maingensublinest}
Under the same hypothesis as in Theorem \ref{mainsublin} we have that, for all $N>1, t\in \bR$,
\be\nonumber
%W_2((\rho(t))_1,\rho^t)
W_2(\rho^t_{N;1},\rho^t)
\leq
C(t)N^{-\frac12}
\ee
with
%\begin{eqnarray}\label{constantc}
%C(t)&:=&
%\left(\int_0^t\|\gen\|_{{L^\infty(supp(\rho_1(s))}}^{2}e^{\int_s^t\L (u)du}ds\right)^{\frac12}\,
%,\nonumber\\
% \L(t)&:=&\sup_{\substack{k,l=1,\dots,d\\(Y,\Xi)\in support[\Phi(t)\#\rho_1^{\otimes N}]}}\Lip(\gen)_{(y_k-y_l,\xi_k-\xi_l)}^2\nonumber
%\end{eqnarray}
%where $\Phi_t$ is the flow   generated by \eqref{eq1}.
%$$
%C(t)
%=
%%(V^\infty+V^1)
%%%e^{2\gen_0s}e^{
%%%(1+2
%%%V^\infty)e^{2\gen_0t}+2\gen_0t}
%%\gen_0V_\infty^2
%%e^{\gen_0(V_\infty+V_1)^2e^{4\gen_0t}+2t}(e^{4\gen_0t}-1)
%%2\gen_0\|v\|_{L^\infty(supp[\rho^{in}])}^2e^{e^{4\gen_0t}4\gen_0
%%(\|v\|_{L^\infty(supp[\rho^{in}])}+
%%\|v\|_{L^1(supp[\rho^{in}])})^2}e^{4t}(e^{4\gen_0t}-1) 
%2\gen_0(\|v\|_{L^\infty(supp[\rho^{in}])}+
%\|v\|_{L^1(supp[\rho^{in}])})
%\left( 2e^{e^{4\gen_0t}4\gen_0
%(\|v\|_{L^\infty(supp[\rho^{in}])}+
%\|v\|_{L^1(supp[\rho^{in}])})^2}e^{4t}(e^{4\gen_0t}-1)
%\right)^{\frac12}
%,
%$$
\begin{eqnarray}
C(t)
&=&
2\gen_0(\|v\|_{L^\infty(supp[\rho^{in}])}+
\|v\|_{L^1(supp[\rho^{in}])})\nonumber\\
&&\times \left( 2e^{\left(e^{4\gen_0t}4\gen_0
(\|v\|_{L^\infty(supp[\rho^{in}])}+
\|v\|_{L^1(supp[\rho^{in}])})^2\right)}e^{4t}(e^{4\gen_0t}-1)
\right)^{\frac12}.
\end{eqnarray}
%where $V_\infty=\|v\|_{L^\infty(supp[\rho^{in}])}$ and $V_1=\|v\|_{L^1(supp[\rho^{in}])}$.
\end{Cor}
\vskip 1cm
Let us finish this section of results by a simple remark.
Flocking properties of the solution of the Vlasov kinetic equation \eqref{vlasovcs} can be derived from the corresponding properties of the particle system \eqref{CS}, as in \cite{haliu,hakimzhang}. But they can also be derived by a direct PDE study of \eqref{vlasovcs}, as in \cite{prt}. This suggest a kind of inverse questioning: suppose one determines some flocking properties for the solution of a general kinetic Vlasov equation, e.g. \eqref{vlasovgencs}. Does this infer on flocking properties for the corresponding particle system, e.g. \eqref{genCS}? Our Corollary \ref{maingensublinest} gives some insight on this problem, as it tells us quantitatively how close is the solution of the kinetic equation (and its tensorial powers) to the marginals of any orders of the pushforward obey the corresponding particle flow of a general $N$ particle density.

More precisely, suppose that the solution $\rho^t$ to \eqref{vlasovgencs} satisfies the following property:
\be\label{suppcs}
supp[\rho^t]\subset B(\bar x+t\bar v,X)\times B(\bar v,Ve^{-\alpha t})
\ee
for some $(\bar x,\bar v)\in\bR^{2d}$ and some positive constants $ X, V,\alpha$. Here $B(\bar w,W)$ designates the ball of center $\bar w$ and radius $W$ in $\bR^d$.
%\begin{eqnarray}\nonumber
%C(t)&:=&
%\left(
%4
%\int_0^t
%%%\|\gen\|_{{L^\infty(supp(\rho_1(s))}}^{2}
%%\sup_{\substack{k,l=1,dots,N\\(Y,\Xi)\in supp[\rho^t_N]}}|\gen(y_k-y_l,\xi_k-\xi_l)|^2
%%e^{\int_s^t\L (u)du}ds
%%\int_0^t%\sup_{\substack{k,l=1,dots,N\\(Y,\Xi)\in supp[\rho^t_N]}}|\gen(y_k-y_l,\xi_k-\xi_l)|^2
%\sup_{(x,v),(x',v')\in supp[\rho^t]}|\gen(x-x',v-v')|^2
%e^{\int_s^tL (u)du}ds
%\right)^{\frac12}\,
%,\nonumber\\
%% \L(t)&:=&2(1+2
%% %\sup_{\substack{k,l=1,\dots,d\\(Y,\Xi)\in support[\Phi(t)\#(\rho^{in})^{\otimes N}]}}\Lip(\gen)_{(y_k-y_l,\xi_k-\xi_l)}^2
%%\sup_{\substack{(x,v),(y,\xi)\in supp[\rho^t]}}\Lip(\gen)_{(x-y,v-\xi)}^2 
%% ).\nonumber
%\b
%\L(t)&:=&
%2(1+2
%\min{(\sup_{\substack{i,l=1,\dots,N\\
%(Y,\Xi)\in supp[\rho^t_N]}}\Lip{(\gen)}^2_{(y_i-y_l,\xi_i-\xi_l)},\sup_{\substack{
%(x,v),(x',v')\in supp[\rho^t]}}\Lip{(\gen)}^2_{(x-x',v-v')})})
%).\nonumber
%\end{eqnarray}

This implies easily that $L(u)$  in Theorem \ref{mainsublin} can be easily  estimated by $$
L(u)\leq 2(1+8\gen_0^2(\bar v^2+V^2e^{-2\alpha u}))
\leq  2(1+8\gen_0^2(\bar v^2+V^2))
$$
and, therefore, in the same Theorem,
\be\label{cdet}
C(t)\leq
4\gen_0(\bar v^2+V^2)\frac{e^{2(1+8(\bar v^2+V^2))t}-1}{2(1+8(\bar v^2+V^2)}.
\ee
Note that in \eqref{cdet}, $C(t)$ has an exponential growth, and not a double exponential one as in Corollary \ref{maingensublinest}. In particular, since \eqref{suppcs} holds true for the Cucker-Smale models by Theorem 3.1, equation $(3.2)$, in \cite{prt}, \eqref{cdet} provides an alternative proof of Theorem \ref{main} with (slightly) different values of the constants involved in its statement.

Of course  the first marginal $\rho_{N;1}^t$ of the pushforward of $(\rho^{in})^{\otimes N}$ by the flows induced by the system \eqref{genCS} has no reason for being compactly supported, as we did not impose anything on the particle flow. Nevertheless,   the following result provides for $\rho_{N;1}^t$ a weak version of the support property \eqref{suppcs}.
\begin{Thm}\label{inverse}
With the same notations and hypothesis as in Theorem \ref{mainsublin}, let us suppose moreover that $\rho^t$ satisfies \eqref{suppcs}.

For all $t,\epsilon>0$ let us define
$$
N_{t,\epsilon}:=\left(\frac{C(t)}\epsilon\right)^2,
$$ where $C(t)$ is the function defined by \eqref{cdet}.

Then, for every $N\geq N_{t,\epsilon}$ and every Lipschitz function $\varphi$ on $\bR^{2d}$ of Lipschitz constant smaller than $1$,
$$
\int\limits_{
%(x,v)\notin 
\bR^{2d}\setminus
B(\bar x+t\bar v,X)\times B(\bar v,Ve^{-\alpha t})}
%\end{Prop}
\varphi(x,v)\rho_{N;1}^tdxdv\leq \epsilon.
$$
\end{Thm}
\begin{Rmk}\label{margplusinv}
Note that, by Remark \ref{margplus}, the same type of result is also true for marginals of any order, by cooking up a new value of $N_{t,\epsilon}$, i.e. $C(t)$, using the constants given in Remark \ref{margplusinv} below:

for all $n=1,\dots,N,\ t,\ \epsilon>0$, there exists $ N_{t,\epsilon,n}$ such that, for every $N\geq N_{t,\epsilon,n}$ and every Lipschitz function $\varphi$ on $\bR^{2dn}$ of Lipschitz constant smaller than $1$,
$$
\int\limits_{
%(x,v)\notin 
\bR^{2dn}\setminus
(B(\bar x+t\bar v,X)\times B(\bar v,Ve^{-\alpha t}))^{\times n}}
%\end{Prop}
\varphi(x,v)\rho_{N;n}^tdxdv\leq \epsilon.
$$
\end{Rmk}
\begin{proof}
It is an immediate consequence of Theorem \ref{mainsublin} or  Theorem \ref{maingensublinest}, and the fact that the Wasserstein distance induces a metric for the weak topology in the sense that
$$
\sup_{\Lip{(\varphi)}\leq 1}\int(\mu-\nu)\varphi(x,v)dxdv\leq W_2(\mu.\nu)
$$
for every probability measures $\mu$ and $\nu$
\end{proof}
\section{Proofs}\label{proofs}
\subsection{Preliminaries}\label{mprelem}
Let us first recall the general situation we are dealing with, in order also to fix the notations.

We consider on $\bR^{2dN}$ the following Cucker-Smale type vector field 
%\cite{CS1,CS2}
\begin{eqnarray}\label{eq1}
\dot {x_i}&=&v_i\\
\dot {v_i}&=&G_i(
%x_i, v_i,\nabla\varphi),
X,V),
\ i=1,\dots,N\nonumber
\end{eqnarray}
where 
\be\label{defG}
G_i(
%y,w,\nabla\varphi)
X,V)
=\frac1N\sum_{j=1}^N\gen(x_i-x_j,v_i-v_j).
%+\nabla_{y}\varphi(t,y,X(t))+F(y),
\ee
Here the  function $\gen(x,v):\bR^{2d}\to\bR^d$ 
 is a   Lipschitz, bounded in $x$ and sublinear in $v$,  continuous function., such that $\gen(x,v),\Lip(\gen)_{(x,v)}\leq \gen_0|v|$.
 
% An example of ``renormalized" $\phi$ functions is defined by 
%$$\phi(v)=\frac v{\sqrt{1+|v|^2}}
%$$

%$X(t)=(x_1(t),\dots,x_N(t))$, $F$ is an external force  and $\varphi$ satisfies the equation
%\be\label{defeqphi}
%\partial_t\varphi(t,y, X(t))=-\Delta_y\varphi-\kappa\varphi +f(t,y,X(t))
%\ee for some $\kappa\geq0$ and function $f$.

%Moreover the two functions $\psi:\bR^d\to\bR$ and $\phi : \bR^b\to\bR^d$ are supposed to be Lipschitz continuous.
We  used the notation $X=(x_1,\dots,x_N),V=(v_1,\dots,v_N)$.

%The case $\phi(y)=y.\ f=\varphi=0$ covers the standard case of Cucker-Smale models.

 In fact, we are rather interested in the Liouville equation associated to \eqref{eq1} \cite{ht}, namely
 \be\label{liouvnd}
 \partial_t\rho^t_N+V\cdot\nabla_X\rho^t_N+
 \sum_{i=1}^N\nabla_{v_i}.(G_i\rho^t_N)=0,\ \rho^{t=0}_N=\rho^{in}_N
\ee

with $\rho^{in}_N\in\mathcal P(\bR^{2dN})$.

Although the argument is standard, let us recall why the solution of \eqref{liouvnd} is equal to the pushforward of $\rho^{in}_N$ by the flow $\Phi_t$  generated by \eqref{eq1}: integrating $\rho^{in}_N$ against a test $\varphi$  function composed by $\Phi^t$ gives
\be\label{liouv3}
\partial_t\int\varphi(\Phi^t(X,V))\rho^{in}_N(X,V)dXdV
=\int\varphi(X,V)\partial_t(\Phi^t\#\rho^{in}_N(X,V))dXdV.
\ee
On the other side
\begin{eqnarray}\label{liouv4}
\partial_t\int\varphi(\Phi^t(X,V))\rho^{in}_N(X,V)dXdV
&=&\int (\dot \Phi^t\cdot\nabla_{(X,V)}\varphi)(\Phi^t(X,V))\rho^{in}_N(X,V)dXdV\nonumber\\
&=&\int((V,G)\cdot\nabla_{(X,V)}\varphi)(\Phi^t(X,V)))\rho^{in}_N(X,V)dXdV\nonumber\\
&=&\int((V,G)\cdot\nabla_{(X,V)}\varphi)(X,V))(\Phi^t\#\rho^{in}_N(X,V)dXdV\nonumber\\
&=&
-\int\varphi(X,V)\nabla_{(X,V)}\cdot\big((V,G)\Phi^t\#\rho^{in}_N(X,V)\big)dXdV,\label{retruc2}
\end{eqnarray}
so that \eqref{liouv3} and \eqref{retruc2} implies that $\Phi^t\#\rho^{in}_N$ solves \eqref{liouvnd}.

We want to prove that the marginals of $\rho^t_N$ tend, as $N\to\infty$,  to   %tensorial powers of 
the solution of a Vlasov type equation.

%\section{The case $f=0$}\label{nodis}
%In this case the Liouville equation reads again 
%\be\label{liouvnd}
%%\nonumber
% \partial_t\rho+v\cdot\nabla_x\rho=
% \sum_{i=1}^NG_i\cdot\nabla_{v_i}\rho
%\ee
%with now
%\be\label{defGnd}
%G_i(t,y,w,\nabla\varphi)=\frac1N\sum_{j=1}^N\phi(w-v_j)\psi(y-x_j)+F(y)
%%+\nabla_{y}\varphi(t,x_i,X(t)),
%\ee
%Note that $\phi$ and $\psi$ being Lipschitz continuous, the Liouville equation has a unique solution.
%\subsection{The associated Vlasov equation}\label{vlasov}
Let us recall that such    Vlasov-type equation associated to \eqref{liouvnd}, introduced in \cite{ht} for the Cucker-Smale model,   reads
\be\label{vlasnd}
\partial_t
\rho^t(x,v)
+v\cdot\nabla_x\rho^t(x,v)+
%G_{\rho_1}.\nabla_v\rho_1
\nabla_v\cdot (G_{\rho^t}\rho^t(x,v))=0
,\ \ \
 \rho^{t=0}=\rho^{in}_1\in L^1(\bR^{2d},dxdv),
\ee
with
\be\label{defgrho}
G_{\rho}(x,v)=\int_{\bR^{2d}}\gen(x-y,v-w)\rho(y,w)dydw.
% +F(x).
\ee
%\section{The result}\label{result}
%Let us start this section by recalling the definition of the second order Wasserstein distance $\MKd$ (see \cite{VillaniAMS,VillaniTOT}).
%\begin{Def}[quadratic Wasserstein distance]\label{defwas}
%The Wasserstein distance of order two between two probability measures $\mu,\nu$ on $\bR^m$ with finite second moments is defined as
%$$
%\MKd(\mu,\nu)^2
%=
%\inf_{\gamma\in\Gamma(\mu,\nu)}\int_{\bR^m\times\bR^m}
%|x-y|^2\gamma(dx,dy)
%$$
%where $\Gamma(\mu,\nu)$ is the set of probability measures on $\bR^m\times\bR^m$ whose marginals on the two factors are $\mu$ and $\nu$.
%\end{Def}
We can now prove the main results of this paper.

\subsection{Proof of Theorem  \ref{mainsublin}}\label{Proofmainsublin}
%The proof will follow directly the proof of Theorem 3.1 in \cite{gmp}. the only difference will be the presence of the external force $F$ and the dependence in velocities of the $2$-body interaction $G_i$.

%\subsection{The dynamics of couplings}

%%%%%%%%%%%%%%%%%%%%%%%%%%%%%%%%%%%%%%%%%%%%%%%%%%%%%%%%%%%%%%%%%%%%%%%%%%%%%%%%%%%%%%%%%%%%%%%%%%%%%%%%%%%%%%%%%%%%%%%
The proof will be articulated around the  four lemmas which follow.

%Let $\pi_N^{in}\in\Pi((\rho^{in})^{\otimes N},(\rho^{in})^{\otimes N})$ satisfy
%\be\lb{SympiN0}
%T_\si\#\pi_N^{in}=\pi_N^{in}\,,\qquad\hbox{ for each }\si\in\fS_N\,,
%\ee
%where $\fS_N$ is the group of permutations of $N$ elements and
%$$
%\ba
%T_\si(x_1,v_1,\ldots,x_N,v_N,y_1,\xi_1,\ldots,y_N,\xi_N)&
%\\
%=(x_{\si(1)},v_{\si(1)},\ldots,x_{\si(N)},v_{\si(N)},y_{\si(1)},\xi_{\si(1)},\ldots,y_{\si(N)},\xi_{\si(N)})&\,.
%\ea
%$$
We will denote $X=(x_1,\dots,x_N),V=(v_1,\dots,v_N),Y=(y_1,\dots,y_N), \Xi=(\xi_1,\dots,\xi_N)$.

Let $\pi^{in}_N$  be defined by
$$
\pi^{in}_N(Y,\Xi,X,V):=(\rho^{in})^{\otimes N}(X,V)\delta(X-Y)\delta(V-\Xi).
$$
Obviously $\pi_N^{in}\in\Pi((\rho^{in})^{\otimes N},(\rho^{in})^{\otimes N})$. Moreover, as mentioned before,
\be\label{inop}
\int_{\bR^{2dN}\times \bR^{2dN}}
(|Y-X|^2+|\Xi-V|^2)
\pi^{in}_N(dY,d\Xi,dX,dV)=0,
\ee
so that  $\pi^{in}_N$ is an optimal coupling between  $(\rho^{in})^{\otimes N}$ and itself.

The following first Lemma will be one of the keys of the proof of   Theorem \ref{main}. It consists in considering in evolving a coupling $\pi^{in}_N$ of two initial conditions of the Liouville \eqref{liouvnd} and Vlasov \eqref{vlasnd}   equations by the two dynamics of each factor of $\pi^{in}_N$.

\begin{Lem}\label{lemone}
Let $\pi_N(t)$ be the unique (measure) solution to the following linear transport equation
\be\label{lintrans}
\partial_t\pi_N+V\cdot\nabla_X\pi_N+\Xi\cdot\nabla_Y\pi_N
+
%\frac1N
\sum_{i=1}^N\left(\nabla_{\xi_i}\cdot(G_i(Y,\Xi)\pi_N)+\nabla_{v_i}\cdot(G_{\rho^t}(x_i,v_i))\pi_N)\right)=0
\ee
with $\pi_N(0)=\pi_N^{in}$.

Then, for all $t\in\bR$, $\pi_N(t)$ is a coupling between $\rho^t_N$ and $(\rho^t)^{\otimes N}$.
% Moreover $\pi_N(t)$ is invariant by permutations $T_\sigma$.
\end{Lem}
\begin{proof}
By taking the two marginals of the two sides of the equality, one gets that they satisfy the two Liouville and Vlasov equations. The result is then obtained by uniqueness of the solutions of both equations. 
\end{proof}
\begin{Lem}\label{lemtwo}
Let 
\begin{eqnarray}
D_N(t)&:=&\int\frac1N
\sum_{j=1}^N(|x_j-y_j|^2+|v_j-\xi_j|^2)d\pi_N(t)\nonumber\\
&=&
\frac1N\int
 ((X-Y)^2+(V-\Xi)^2)
d\pi_N(t)\,.\nonumber
\end{eqnarray}
Then
$$
\ba
\frac{dD_N}{dt}\le{\bf L}(t) D_N+\frac{1}N\sum_{j=1}^N\int\left|
%\grad V_{\rho[f]}(x_j)-\frac1{N}\sum_{k=1}^N\grad V(x_j-x_k)
G_{\rho^t}(x_i,v_i)-G_i(X,V)
\right|^2
%\pi_N
{(\rho^t)^{\otimes N}}dXdV
&\,,
\ea
$$
with
\begin{eqnarray}
 {\bf L}(t):=&&
%&:=&
%2(1+2\Lip(\grad V)^2)\,
\label{defl}\\
2(1+2
\min{(\sup_{\substack{i,l=1,\dots,N\\
X,V,Y,\Xi\in supp[\pi_N(t)]}}\Lip{(\gen)}^2_{(x_i-x_l,v_i-v_l)},\sup_{\substack{i,l=1,\dots,N\\
X,V,Y,\Xi\in supp[\pi_N(t)]}}\Lip{(\gen)}^2_{(y_i-y_l,\xi_i-\xi_l)}})&&
\nonumber
%\sup_{\substack{k,l=1,\dots,N\\((X,V),(Y,\Xi))\in support[\pi_N(t)]}}\Lip(\gen)_{(y_k-y_l,\xi_k-\xi_l)}^2
%\sup_{\substack{i=1,dots,N\\(Y,\Xi:X,V)\in supp[\pi_N(t)]\\((y,\xi)\in supp[\rho^t]}}\Lip(\gen)_{(x_i-y,\xi_i-\xi)}^2
).
\end{eqnarray}
\end{Lem}
\begin{proof}
We first notice that
$$
\frac{dD_N}{dt}
=\frac2N\int
\left((V-\Xi).(X-Y)+
%\frac1N
\sum_{i=1}^N(v_i-\xi_i).(G_{\rho^t}(x_i,v_i)-G_i(Y,\Xi))\right)d\pi_N.
$$
Using $2uv\leq u^2+v^2$ we get
\begin{eqnarray}\label{trucenplume}
\frac{dD_N}{dt}
&\leq&\frac1{N}\int\left((X-Y)^2+2(V-\Xi)^2)
+
%\frac1{N}
\sum_{i=1}^N
|G_i(Y,\Xi)-G_{\rho^t}(x_i,v_i)|^2\right)d\pi_N
\nonumber\\
&\leq& 
2D_N(t)
+\frac1{N}\int\left(
%\frac1N
\sum_{i=1}^N
|G_i(Y,\Xi)-G_{\rho^t}(x_i,v_i)|^2\right)d\pi_N.
\end{eqnarray}
%Let us rewrite the conclusion of Lemma \ref{lemtwo} as 
%$$
%\ba
%\frac{dD_N}{dt}\le\L D_N+\frac{1}N\sum_{j=1}^N\int\left|
%%\grad V_{\rho[f]}(x_j)-\frac1{N}\sum_{k=1}^N\grad V(x_j-x_k)
%G_{\rho^t}(x_i,v_i)-
%%G_i(X,V)
%\frac1N\sum_{j=1}^N
%\frac{\xi_i-\xi_j}{\sqrt{1+|\xi_i-\xi_j|^2}}\phi(y_i-y_j)
%\right|^2d
%%\pi_N
%(\rho_1(t))^{\otimes N}
%&\,,
%\ea
%$$
Let us 
%introduce in the square inside the last integral 
add to $G_i(Y,\Xi)-G_{\rho^t}(x_i,v_i)$
the null term 
$
\big(G_i(X,V)-G_i(X,V)\big)
$
so that
$$
|G_i(Y,\Xi)-G_{\rho^t}(x_i,v_i)|^2\leq 
2\big(|G_i(Y,\Xi)-G_i(X,V)|^2+|G_i(X,V)-G_{\rho^t}(x_i,v_i)|^2\big)
$$
Let us first  estimate
\begin{eqnarray}
&&|G_i(Y,\Xi)-G_i(X,V)|^2\leq
%\nonumber\\
%&&
\min{(\Lip(G_i)_{(X,V)}^2,\Lip(G_i)_{(Y,\Xi)}^2)}\big(|X-Y|^2+|V-\Xi|^2\big)\nonumber.
\end{eqnarray}
By \eqref{defG}, we have
\begin{eqnarray}
&&\min{(\Lip(G_i)_{(X,V)}^2,\Lip(G_i)_{(Y,\Xi)}^2)}\nonumber\\
&\leq&
\min{(\sup_{l=1,\dots,N}\Lip{(\gen(x_i-x_l,v_i-v_l)}^2,\sup_{l=1,\dots,N}\Lip{(\gen(y_i-y_l,\xi_i-\xi_l)}^2)}.\nonumber
\end{eqnarray}
%&=&
%|\int (\gen(y_i-y,\xi_i-\xi)-\gen(x_i-y,v_i-\xi))\rho^t(y,\xi)dyd\xi|\nonumber\\
%&\leq&
%\int |(x_i-y_i,v_i-\xi_i)|\Lip(\gen)_{(x_i-y,\xi_i-\xi)}
%\rho^t(y,\xi)dyd\xi\nonumber\\
%&\leq&
%\sup_{(y,\xi)\in supp[\rho^t]}\Lip(\gen)_{(x_i-y,\xi_i-\xi)}
%\int |(x_i-y_i,v_i-\xi_i)|
%\rho^t(y,\xi)dyd\xi\nonumber\\
%&=&
%\sup_{(y,\xi)\in supp[\rho^t]}\Lip(\gen)_{(x_i-y,\xi_i-\xi)}
% |(x_i-y_i,v_i-\xi_i)|\nonumber
%\end{eqnarray}
Therefore, by convexity,
\begin{eqnarray}
&&\frac2{N}\int\sum_{i=1}^N
|G_i(Y,\Xi)-G_i(X,V)|^2
\pi_N(dX,dV,dY,d\Xi)\nonumber\\
&\leq &
4\min{(\sup_{\substack{i,l=1,\dots,N\\
X,V,Y,\Xi\in supp[\pi_N(t)]}}\Lip{(\gen(x_i-x_l,v_i-v_l)}^2,\sup_{\substack{i,l=1,\dots,N\\
X,V,Y,\Xi\in supp[\pi_N(t)]}}\Lip{(\gen(y_i-y_l,\xi_i-\xi_l)}^2)}
\nonumber\\
&&\times D_N(t).\nonumber
\end{eqnarray}

And the lemma follows by \eqref{trucenplume}.
% and Lemma \ref{lemone}.
%%Then 
%%$$
%%\left|\psi{(v_i-v_j)}
%%%{\sqrt{1+|v_i-v_j|^2}}\phi(x_i-x_j)
%%-
%%\psi{(\xi_i-\xi_j)}
%%%{\sqrt{1+|\xi_i-\xi_j|^2}}
%%\phi(y_i-y_j)
%%\right|^2
%%$$ 
%%can be approximated by 
%$$
%|G_i(Y,\Xi)-G_i(X,V)|\leq
%|(v_i-v_j)-(\xi_i-\xi_j)|\Lip(\psi)^2
%+
%(x_i-x_j-(y_i-y_j))^2)\Lip(\phi)^2
%$$
%$$
%\leq 4(\Lip(\phi)^2+\Lip(\psi)^2)((x_i-y_i)^2+(v_i-\xi_i)^2+(x_j-y_j)^2+(v_j-\xi_j)^2)
%$$
%This part gives the $4(\Lip(\phi)^2+\Lip(\psi)^2)$ part in $\L$, and since the remaining part to integrate contains only the $(X,V)$ variable, the integral against $\pi_N$ can be replace by the one against $\rho_1(t)^{\otimes N}$.
\end{proof}
\vskip 1cm
It remains to estimate in \eqref{trucenplume} the term
\begin{eqnarray}
&&\frac{1}N\int
\sum\limits_{i=1}^N
|G_i(X,V)-G_{\rho^t}(x_i,v_i)|^2
\pi^t_N(dX,dV,dY,d\Xi)\nonumber\\
&=&
\frac{1}N\int
\sum\limits_{i=1}^N
|G_i(X,V)-G_{\rho^t}(x_i,v_i)|^2
%\pi_N
(\rho^t)^{\otimes N}dXdV.\nonumber
\end{eqnarray}

The following result is a variant of  Lemma 3.3 in \cite{gmp} with the special value $p=2$ and $d$ replaced by $2d$.
\begin{Lem}\label{lemthree}
Let  
%$\rho_N$ and 
$\rho$ be  a compactly supported probability density on 
%$\bR^{2dN}$ and 
$\bR^{2d}$ 
%respectively,   
and  let $F$ be a  locally bounded vector field on $\bR^{2d}$.
%, bounded on the support of $\rho$. 

For each  $j=1,\ldots,N$, one has
\begin{eqnarray}\label{eqlemthree}
&&\int\left|F\star\rho(x_j,v_j)-\frac1N\sum_{k=1}^NF(x_j-x_k,v_j-v_k)\right|^{2}\rho^{\otimes N}dXdV\\
&&\le
\frac{4}{N}\sup_{\substack{
%k=1,dots,N\\(Y,\Xi)\in supp[\rho_N]
%\\
(x,v),(x',v')\in supp[\rho]}}|F(x-x',v-v')|^2.\nonumber
\end{eqnarray}
\end{Lem}
\begin{proof}
Let us denote, for $x_j,x,v_j,v\in\bR^d, j=1,\dots,N$,
\be\label{defnu}
\nu_{x_j,v_j}(x,v):=
F\star\rho(x_j,v_j)-F(x_j-x,v_j-v)
\ee
(note that $F\star\rho(x_j,v_j)$ doesn't depend on $(x,v)$).

One has (let us remind the notation $X:=(x_1,\dots,x_N),\ V:=(v_1,\dots,v_N)$)
%,\ x_j,v_j\in\bR^d,\ j=1,\dots,N$)
\begin{eqnarray}
&&\int\left|F\star\rho(x_j,v_j)-\frac1N\sum_{k=1}^NF(x_j-x_k,v_j-v_k)\right|^{2}\rho^{\otimes N}dXdV\nonumber\\
&=&
\int\left|\frac1N\sum_{k=1}^N\nu_{x_j,v_j}(x_k,v_k)\right|^{2}\rho^{\otimes N}dXdV\nonumber\\
&=&
\frac1{N^2}
\sum_{k,l=1,\dots N}\int \nu_{x_j,v_j}(x_k,v_k)\nu_{x_j,v_j}(x_l,v_l)\rho^{\otimes N}dXdV\nonumber\\
&=&
%\frac1{N^2}
%\sum_{k,l=1,\dots N}\int \nu(x_k,v_k)\nu(x_l,v_l)\rho(x_k,v_k)dx_kdv_k\rho(x_l,v_l)dx_ldv_l\prod_{k\neq m\neq l}\rho(x_m,v_m)dx_mdv_m\nonumber\\
%&=&
\frac1{N^2}
\sum_{k\neq l\neq j\neq k}\int\left(\int \nu_{x_j,v_j}(x_k,v_k)\rho(x_k,v_k)dx_kdv_k\right)\nu_{x_j,v_j}(x_l,v_l)\rho(x_l,v_l)\rho(x_j,v_j)dx_ldv_ldx_jdv_j\nonumber\\
&+&
\frac1{N^2}
\sum_{k\neq l= j}\int\left(\int \nu_{x_j,v_j}(x_k,v_k)\rho(x_k,v_k)dx_kdv_k\right)\nu_{x_j,v_j}(x_j,v_j)\rho(x_j,v_j)dx_jdv_j\nonumber\\
&+&
\frac1{N^2}\sum_{k\neq j}
\int \nu_{x_j,v_j}(x_k,v_k)^2\rho(x_k,v_k)dx_kdv_k\rho(x_j,v_j)dx_jdv_j\nonumber\\
&+&
\frac1{N^2}
\int \nu_{x_j,v_j}(x_j,v_j)^2\rho(x_j,v_j)dx_jdv_j\nonumber\\
%&=&
%\frac1{N^2}\sum_{k\neq j}
%\int \nu_{x_j,v_j}(x_k,v_k)^2\rho(x_k,v_k)dx_kdv_k\rho(x_j,v_j)dx_jdv_j+
%\frac1{N^2}\int \nu_{x_j,v_j}(x_j,v_j)^2\rho(x_j,v_j)dx_jdv_j\nonumber\\
&\leq &\frac{N+1}{N^2}\sup_{(x,v),(x',v')\in supp[\rho]}\nu_{x',v'}(x,v)^2\nonumber
\end{eqnarray}
since,  by \eqref{defnu},
$
\int\nu_{x',v'}(x,v)\rho(x,v)dxdv=0. 
$ for all $x',v'\in \bR^d$, and $\int \rho(x,v)dxdv=1$.
\vskip 0.5cm

By \eqref{defnu} again, 
$$|\nu_{x',v'}(x,v)|\leq 2\sup\limits_{\substack{
(x,v)\in supp[\rho]}}|F(x'-x,v'-v)|
\leq 
 2\sup\limits_{\substack{
(x,v),(x',v')\in supp[\rho]}}|F(x'-x,v'-v)|
$$ and the lemma is proved.
\end{proof}
% There are two differences between the statements in Lemma \ref{lemthree} and Lemma 3.3. in \cite{gmp}: 
% 
% -  the term $\sup\limits_{\substack{k,l1,dots,N\\(Y,\Xi)\in supp[\rho^t_N]}}|F(y_k-y_l,\xi_k-\xi_l)|^2$ used  in the right hand side of \eqref{eqlemthree}  instead of the norm $L^\infty(\bR^{2d})$ used in \cite{gmp}.This difference is unessential since the integration in the right hand side of \eqref{eqlemthree} is effectively over $supp[\rho^t_N]$.
% 
% - the measure $\rho^t_n$ in the left hand-side of \eqref{eqlemthree}  instead of the measure $\rho^{\otimes N}$ used in \cite{gmp}. It is clear form the (beginning) of the proof of Theorem 3.3. in \cite{gmp} that it works excatly the same way with $\rho^t_N$.
%\end{proof}
Lemma \ref{lemthree} with $F(x,v)=\gen(x,v)$ together with Lemma \ref{lemtwo} gives immediately that
$$
\ba
\frac{dD_N}{dt}\le  L (t) D_N+\frac4N
%\|\gen|_{L^\infty(supp(\rho^t)}^{2}\,
\sup_{(x,v),(x',v')\in supp[\rho^t]}|\gen(x-x',v-v')|^2
\ea
$$
and, by Gronwall's inequality,
\be\lb{IneqD}
D_N(t)\le 
%D_N(0)e^{\int_0^tL(s)ds }+
\frac4N \int_0^t%\sup_{\substack{k,l=1,dots,N\\(Y,\Xi)\in supp[\rho^t_N]}}|\gen(y_k-y_l,\xi_k-\xi_l)|^2
\sup_{(x,v),(x',v')\in supp[\rho^t]}|\gen(x-x',v-v')|^2
e^{\int_s^tL (u)du}ds\,
\ee
since $D_N(0)=0$ by \eqref{inop}.

Let us denote by $(\pi_N(t))_1$ the measure on $\bR^{2d}\times \bR^{2d}$ defined, for every test function $\varphi(x_1,v_1;y_1,\xi_1)$, by
$$
\int _{\bR^{2dN}\times\bR^{2dN}}\varphi\pi_N(dX,dV;dYd\Xi)
=
\int_{\bR^{2d}\times\bR^{2d}}\varphi(x_1,v_1;y_1,\xi_1)(\pi_N(t))_1(dx_1,dv_1;dy_1,d\xi_1).
$$

We now notice the following 
%obvious 
straightforward fact.

\begin{Lem}\label{lemfour}
$(\pi_N(t))_1$ is a coupling between $\rho^t_{N;1}$ and $\rho^t$.
\end{Lem}
Let us note that $\pi_N^{in}$ is obviously symmetric by permutation of the phase-space variables, that is
\be\lb{SympiN0}
T_\si\#\pi_N^{in}=\pi_N^{in}\,,\qquad\hbox{ for each }\si\in\fS_N\,,
\ee
where $\fS_N$ is the group of permutations of $N$ elements and
$$
\ba
T_\si(x_1,v_1,\ldots,x_N,v_N,y_1,\xi_1,\ldots,y_N,\xi_N)&
\\
=(x_{\si(1)},v_{\si(1)},\ldots,x_{\si(N)},v_{\si(N)},y_{\si(1)},\xi_{\si(1)},\ldots,y_{\si(N)},\xi_{\si(N)})&\,.
\ea
$$
Therefore, $\pi_N(t)$ is also symmetric by permutations for all $t\in\bR$, as being  the unique solution to  the equation \eqref{lintrans}, being itself, by construction, symmetric by permutations, .

Consequently,  one has easily that
\be\lb{MinDNp}
D_N(t)=\int(|x_1-y_1|^2+|v_1-\xi_1|^2)d(\pi_N(t))_1
\ee
and Lemma \ref{lemfour} immediately implies that
\be
D_N(t)\ge W_2(\rho^t_{N;1},\rho^t)^2\,,
\ee
so that, by  \eqref{IneqD},
\be\label{oufff}
W_2(\rho^t_{N;1},\rho^t)^2
\leq
\frac4N \int_0^t%\sup_{\substack{k,l=1,dots,N\\(Y,\Xi)\in supp[\rho^t_N]}}|\gen(y_k-y_l,\xi_k-\xi_l)|^2
\sup_{(x,v),(x',v')\in supp[\rho^t]}|\gen(x-x',v-v')|^2
e^{\int_s^tL (u)du}ds.
\ee
%Remember that \eqref{oufff} is valid for $D(0)$ as defined in Lemma \ref{lemtwo} for $any$  $\pi_N$ coupling $(\rho_1^{in})^{\otimes N}$ with itself. Since $W_2$ is a distance, one has, for any of these $\pi^{op}$, 
%$$
%W_2((\rho_1^{in})^{\otimes N},(\rho_1^{in})^{\otimes N})^2=\int((X-Y)^2+(V-\Xi)^2d\pi_N^{op}=ND_N(0)=0
%$$
%%{  for some  optimal coupling }$\pi_N^{op}$.
%so that
%$$
%W_2(\rho_1(t),(\rho(t))_1)^2
%\leq
%%D_N(0)e^{\int_0^tL(s)ds }+
%%\frac{16}N \int_0^t
%%%\|\gen\|_{{L^\infty(supp(\rho_1(s))}}^{2}
%%\sup_{\substack{k,l1,dots,N\\(Y,\Xi)\in supp[\rho^t_N]}}|\gen(y_k-y_l,\xi_k-\xi_l)|^2
%%e^{\int_s^tL (u)du}ds
%\frac4N \int_0^t%\sup_{\substack{k,l=1,dots,N\\(Y,\Xi)\in supp[\rho^t_N]}}|\gen(y_k-y_l,\xi_k-\xi_l)|^2
%\sup_{(x,v),(x',v')\in supp[\rho^t]}|\gen(x-x',v-v')|^2
%e^{\int_s^tL (u)du}ds\,.
%$$
Finally, recalling that
$$
\pi^{in}_N(Y,\Xi,X,V):=(\rho^{in})^{\otimes N}(X,V)\delta(X-Y)\delta(V-\Xi),
$$
one gets immediately that
\begin{eqnarray}
(Y,\Xi;X,V)\in supp[\pi_N(t)]&\Rightarrow& (Y,\Xi)\in supp[\rho^t_N]\mbox{ and }\nonumber\\
&\Rightarrow& (X,V)\in supp[(\rho^t)^{\otimes _N}]
\Leftrightarrow (x_i,v_i)\in supp[\rho^t],\ i=1,\dots,N.\nonumber
%\Rightarrow (x_i,v_i)\in supp[\rho^t],\ i=1,\dots,N.
\end{eqnarray}
Therefore, after \eqref{defl}
%$$
%L(t)=\L(t):=2(1+2\sup_{\substack{(x,v),(y,\xi)\in supp[\rho^t]}}\Lip(\gen)_{(x-y,v-\xi)}^2).
%$$
%where $\Phi(t)$ is the flow   generated by \eqref{eq1}.
\begin{eqnarray}
{\bf L}(t)\leq L(t):=&&
\label{defldef}\\
2(1+2
\min{(\sup_{\substack{i,l=1,\dots,N\\
(Y,\Xi)\in supp[\rho^t_N]}}\Lip{(\gen)}^2_{(y_i-y_l,\xi_i-\xi_l)},\sup_{\substack{
(x,v),(x',v')\in supp[\rho^t]}}\Lip{(\gen)}^2_{(x-x',v-v')})})&&
).\nonumber
\end{eqnarray}

%Choosing now this coupling for the definition of $D_N(0)$ in Lemma \ref{lemtwo} we get that
%$$
%D_N(0)=\frac1NW_2((\rho_1^{in})^{\otimes N},(\rho_1^{in})^{\otimes N})^2=0
%$$
%so that \eqref{oufff} gives the result. 
Theorem \ref{mainsublin} is proven.

\begin{Rmk}[Higher order Wasserstein and marginals]\label{highorders} 

As we wanted to leap this note as short as possible, we expressed our results only for the first marginal $\rho_{N;1}$ in the  $2$-Wasserstein topology, but the method developed in \cite{gmp} allows as well, with the same kind of modification than the ones used before in this section, to the higher cases.
Let us remind, for $p\geq 1$, the definition of $W_p(\mu,\nu)$ for two positive measures $\mu,\nu$ (cf. Definition \ref{defwas}
$$
W_p(\mu,\nu)^p=\inf_{\gamma\in\Gamma(\mu,\nu)}\int_{\bR^m\times\bR^m}
|x-y|^p\gamma(dx,dy),
$$
and, for any probability measure $\rho$ on $\bR^{2dN}$ and $n=1,\dots,N$, the definition of the $n$th marginal $\rho_{N;n}$ defined on $\bR^{2d{n}}$
\begin{eqnarray}
&&\rho_{N;n}(x_1,\dots,x_{n},v_1,\dots,v_{n}):=\nonumber\\
&&
\int_{\bR^{2d(N-n)}}\rho((x_1,\dots,x_{n}, x_{n+1},\dots,x_N,v_1,\dots,v_{n},v_{n+1},\dots,v_N)dx_{n+1}\dots dx_Ndv_{n+1}\dots dv_N.\nonumber
\end{eqnarray}
One gets, for each $p\geq 1$, $N\geq 1$ and $n=1,\dots,N$,
$$
\frac1nW_p(\rho^t_{N;n},(\rho^t)^{\otimes n})^p\leq D_{p,n}(t)N^{-\min{(p/2,1)}}
$$
with
$$
D_{p,n}(t)= 2^{2p}\max{(1,p-1)}([p/2]+1)\int_0^t
%\|\gen\|_{L^\infty(supp(\rho^s))}^
\sup_{\substack{k,l=1,\dots,N\\(Y,\Xi)\in supp[\rho^t_N]}}|\gen(y_k-y_l,\xi_k-\xi_l)|^
pe^{
2\max{(1,p-1)}\int_s^t
%(1+2^{p-1}...
L(u)
)du}ds
%.2^nK_p([p/2]+1)\dots
$$
where 
%$K_p=\max(1,p-1),\Lambda_p=K_p2(1+2^{p-1}\Lip(\psi))$.
$$
L(u)=
1+2^{p-1}
%\sup_{\substack{k,l=1,dots,N\\((X,V),(Y,\Xi))\in support[\pi_N(t)]}}\Lip(\gen)_{(y_k-y_l,\xi_k-\xi_l)}^
\sup_{\substack{(x,v),(y,\xi)\in supp[\rho^t]}}\Lip(\gen)_{(x-y,v-\xi)}^
p.
$$
%\sup_{\substack{k,l=1,dots,N\\(Y,\Xi)\in supp[\rho^t_N]}}|\gen(y_k-y_l,\xi_k-\xi_l)|^2
%e^{\int_s^t\L (u)du}ds\right)^{\frac12}\,
%,\nonumber\\
% \L(t)&:=&2(1+2
% %\sup_{\substack{k,l=1,\dots,d\\(Y,\Xi)\in support[\Phi(t)\#(\rho^{in})^{\otimes N}]}}\Lip(\gen)_{(y_k-y_l,\xi_k-\xi_l)}^2
%\sup_{\substack{(x,v),(y,\xi)\in supp[\rho^t]}}\Lip(\gen)_{(x-y,v-\xi)}^2 
% )
The main changes in the proof are the use of the Young inequality
$$
puv^{p-1}\leq u^p+(p-1)v^p\leq\max(1,p-1)(u^p+v^p)
$$ 
instead of $2uv\leq u^2+v^2$ before \eqref{trucenplume}, the convexity of $|\cdot|^p$ for $p\geq 1$ and a variant of Lemma \ref{lemthree}, similar to  Lemma 3.3 in \cite{gmp}, which reads
\begin{eqnarray}
&&\int\left|F\star\rho(x_j,v_j)-\frac1N\sum_{k=1}^NF(x_j-x_k,v_i-v_j)\right|^{p}
%\prod_{m=1}^N\rho(x_m,v_m)dx_mdv_m
\rho^{\otimes N}dXdV\nonumber\\
&\leq&
\frac{2[p/2]+2}{N^{\min(p/2,1)}}
\sup_{\substack{k,l1,\dots,N\\(Y,\Xi)\in supp[\rho^t_N]}}|F(y_k-y_l,\xi_k-\xi_l)|^{p}.\nonumber
\end{eqnarray}
Finally, the statement of Lemma \ref{lemfour} becomes  now easily that
$\pi_N(t)_n$ is a coupling between $\rho^t_{N;n}$ and $(\rho^t)^{\otimes n}$, where $\pi_N(t)_n$ is defined through by
\begin{eqnarray}
&&\int _{\bR^{2dN}\times\bR^{2dN}}\varphi\pi_N(dX,dV;dYd\Xi)
%\nonumber\\
%&=&
=
\int_{\bR^{2d}\times\bR^{2d}}\varphi
%(
%%x_1,v_1;y_1,\xi_1
%(x,v;y,\xi)_1,\dots,(x,v;y,\xi)_n
%)
\pi_N(t)_n(d(x,v;y,\xi)_1\dots d(x,v;y,\xi)_n)\nonumber
\end{eqnarray}
 for every test function $\varphi((x,v;y,\xi)_1,\dots,(x,v;y,\xi)_n)$,

%We now notice the following obvious fact.
%
%\begin{Lem}\label{lemfour}
%$(\pi_N(t))_1$ is a coupling between $\rho^t_{N;1}$ and $\rho^t$.
%\end{Lem}
\end{Rmk}
\subsection{Proof of Corollary \ref{maingensublinest}}\label{proofcor2}
We get to estimates $L(t)$ and $C(t)$
%, i.e. $\|\gen|\|_{L^\infty(supp[\rho_1(t)])}$ 
as given in Theorem \ref{mainsublin} out of the estimates established in Appendices \ref{dynestpar} and \ref{dynestpkin}.

Let us recall that
\begin{eqnarray}\nonumber
%C(t)&:=&
%\left(
%4
%\int_0^t
%%%\|\gen\|_{{L^\infty(supp(\rho_1(s))}}^{2}
%%\sup_{\substack{k,l=1,dots,N\\(Y,\Xi)\in supp[\rho^t_N]}}|\gen(y_k-y_l,\xi_k-\xi_l)|^2
%%e^{\int_s^t\L (u)du}ds
%%\int_0^t%\sup_{\substack{k,l=1,dots,N\\(Y,\Xi)\in supp[\rho^t_N]}}|\gen(y_k-y_l,\xi_k-\xi_l)|^2
%\sup_{(x,v),(x',v')\in supp[\rho^t]}|\gen(x-x',v-v')|^2
%e^{\int_s^tL (u)du}ds
%\right)^{\frac12}\,
%,\nonumber\\
% \L(t)&:=&2(1+2
% %\sup_{\substack{k,l=1,\dots,d\\(Y,\Xi)\in support[\Phi(t)\#(\rho^{in})^{\otimes N}]}}\Lip(\gen)_{(y_k-y_l,\xi_k-\xi_l)}^2
%\sup_{\substack{(x,v),(y,\xi)\in supp[\rho^t]}}\Lip(\gen)_{(x-y,v-\xi)}^2 
% ).\nonumber
L(t)&:=&
2(1+2
\min{(\sup_{\substack{i,l=1,\dots,N\\
(Y,\Xi)\in supp[\rho^t_N]}}\Lip{(\gen)}^2_{(y_i-y_l,\xi_i-\xi_l)},\sup_{\substack{
(x,v),(x',v')\in supp[\rho^t]}}\Lip{(\gen)}^2_{(x-x',v-v')})})
).\nonumber\\
&\leq&
2(1+2
\sup_{\substack{
(x,v),(x',v')\in supp[\rho^t]}}\Lip{(\gen)}^2_{(x-x',v-v')}
).\nonumber\\&\leq&
2(1+2\gen_0^2
\sup_{\substack{
(x,v),(x',v')\in supp[\rho^t]}}|v-v'|^2)\nonumber\\
&\leq &
2(1+8\gen_0^2\sup_{\substack{
(x,v)\in supp[\rho^t]}}|v|^2)\nonumber\\
&\leq &
2(1+8\gen_0^2\sup_{\substack{
(x,v)\in supp[\rho^{in}]}}|\Phi_v(t)(x,v)|^2).\nonumber
\end{eqnarray}
%\be\nonumber
%\L(t):=
%%\sup_{\substack{k,l=1,\dots,d\\(Y,\Xi)\in support[\Phi(t)\#\rho_1^{\otimes N}]}}\Lip(\gen)_{(y_k-y_l,\xi_k-\xi_l)}^2
%2(1+2
% %\sup_{\substack{k,l=1,\dots,d\\(Y,\Xi)\in support[\Phi(t)\#(\rho^{in})^{\otimes N}]}}\Lip(\gen)_{(y_k-y_l,\xi_k-\xi_l)}^2
%\sup_{\substack{(x,v),(y,\xi)\in supp[\rho^t]}}\Lip(\gen)_{(x-y,v-\xi)}^2 
% )
%=
%2(1+2
% %\sup_{\substack{k,l=1,\dots,d\\(Y,\Xi)\in support[\Phi(t)\#(\rho^{in})^{\otimes N}]}}\Lip(\gen)_{(y_k-y_l,\xi_k-\xi_l)}^2
%\|\psi\|^2_\infty\sup_{\substack{(x,v),(y,\xi)\in supp[\rho^t]}}|v-\xi|^2 
% )
%\ee 

Thanks to \eqref{estiphi}, we get
\begin{eqnarray}
L(t)
%&:=&
%2(1+2
% %\sup_{\substack{k,l=1,\dots,d\\(Y,\Xi)\in support[\Phi(t)\#(\rho^{in})^{\otimes N}]}}\Lip(\gen)_{(y_k-y_l,\xi_k-\xi_l)}^2
%\sup_{\substack{(x,v),(y,\xi)\in supp[\rho^t]}}\Lip(\gen)_{(x-y,v-\xi)}^2 
% )\nonumber\\
%&\leq&
%2(1+2
% %\sup_{\substack{k,l=1,\dots,d\\(Y,\Xi)\in support[\Phi(t)\#(\rho^{in})^{\otimes N}]}}\Lip(\gen)_{(y_k-y_l,\xi_k-\xi_l)}^2
%\sup_{\substack{(x,v),(y,\xi)\in supp[\rho^t]}}\gen_0^2|v-\xi|^2 
% )
% \leq 4(1+2
% %\sup_{\substack{k,l=1,\dots,d\\(Y,\Xi)\in support[\Phi(t)\#(\rho^{in})^{\otimes N}]}}\Lip(\gen)_{(y_k-y_l,\xi_k-\xi_l)}^2
%\sup_{\substack{(x,v)\in supp[\rho^t]}}\gen_0^2|v|^2 
% )\nonumber\\
%&=&
%4(1+2
% %\sup_{\substack{k,l=1,\dots,d\\(Y,\Xi)\in support[\Phi(t)\#(\rho^{in})^{\otimes N}]}}\Lip(\gen)_{(y_k-y_l,\xi_k-\xi_l)}^2
%\sup_{\substack{(x,v)\in supp[\rho^{in}]}}\gen_0^2|\Phi(t)(x,v)|^2 
% )\nonumber\\
 &\leq&
 2(1+16\gen_0^2e^{4\gen_0t}(\|v\|_{L^\infty(supp[\rho^{in}])}+
\|v\|_{L^1(supp[\rho^{in}])})^2)\nonumber
\end{eqnarray}

%Moreover, \eqref{estpart} gives
%\begin{eqnarray}
%\sup_{\substack{k,l=1,\dots,N\\(Y,\Xi)\in supp[\rho^t_N]}}|\gen(y_k-y_l,\xi_k-\xi_l)|^2&\leq &2\gen_0^2
%%e^{2\gen_0t}(\|v\|_{L^\infty(supp[\rho^{in}])}+
%%\|v\|_{L^1(supp[\rho^{in}])}),
%\sup_{\substack{k=1,\dots,N\\(Y,\Xi)\in supp[\rho^t_N]}}
%|\xi_k|^2
%\nonumber\\
%&=&
%2\gen_0^2
%\sup_{\substack{k=1,\dots,N\\(X,V)\in supp[(\rho^{in})^{\otimes N}]}}|v_k(t)|^2\nonumber\\
%&\leq& 2\gen_0^2
%\sup_{\substack{k=1,\dots,N\\(X,V)\in supp[(\rho^{in})^{\otimes N}]}}|v_k|^2e^{4\gen_0t}\nonumber\\
%&=&
%2\gen_0^2
%\|v\|_{L^\infty(supp[\rho^{in}])}^2e^{4\gen_0t}\nonumber
%\end{eqnarray}
%(here $(X(t),V(t))$ solves \eqref{genCS}.

By the same argument, we get
\begin{eqnarray}
C(t)^2&:=&
4
\int_0^t
\sup_{(x,v),(x',v')\in supp[\rho^t]}|\gen(x-x',v-v')|^2
e^{\int_s^tL (u)du}ds
%4
%\int_0^t
%%\|\gen\|_{{L^\infty(supp(\rho_1(s))}}^{2}
%\sup_{\substack{k,l=1,dots,N\\(Y,\Xi)\in supp[\rho^t_N]}}|\gen(y_k-y_l,\xi_k-\xi_l)|^2
%e^{\int_s^t\L (u)du}ds\,
,\nonumber\\
&\leq&
32\gen_0^2
(\|v\|_{L^\infty(supp[\rho^{in}])}+
\|v\|_{L^1(supp[\rho^{in}])})^2\nonumber\\
&&
\times \int_0^t e^{4(\gen_0s+t-s)}e^{16\gen_0^2(\|v\|_{L^\infty(supp[\rho^{in}])}+
\|v\|_{L^1(supp[\rho^{in}])})^2(e^{4\gen_0t}-e^{4\gen_0s})/4\gen_0}ds
\nonumber\\
&\leq &
8\gen_0^2(\|v\|_{L^\infty(supp[\rho^{in}])}+
\|v\|_{L^1(supp[\rho^{in}])})^2e^{e^{4\gen_0t}4\gen_0
(\|v\|_{L^\infty(supp[\rho^{in}])}+
\|v\|_{L^1(supp[\rho^{in}])})^2}e^{4t}(e^{4\gen_0t}-1).\nonumber
\end{eqnarray}

Corollary \ref{maingensublinest} is proven.

\subsection{Proof of Theorem \ref{main}}\label{Proofmain}
%\end{proof}
%\Large
In order to prove Theorem \ref{main}, we need to estimate, in the Cucker-Smale  particular case, that is when
$$
\gen(x,v)=\psi(x)v,
$$
the two quantities
%\begin{eqnarray}\nonumber
%C(t)&:=&
%\left(
%4
%\int_0^t
%%\|\gen\|_{{L^\infty(supp(\rho_1(s))}}^{2}
%\sup_{\substack{k,l=1,dots,N\\(Y,\Xi)\in supp[\rho^t_N]}}|\gen(y_k-y_l,\xi_k-\xi_l)|^2
%e^{\int_s^t\L (u)du}ds\right)^{\frac12}\,
%,\nonumber\\
% \L(t)&:=&2(1+2
% %\sup_{\substack{k,l=1,\dots,d\\(Y,\Xi)\in support[\Phi(t)\#(\rho^{in})^{\otimes N}]}}\Lip(\gen)_{(y_k-y_l,\xi_k-\xi_l)}^2
%\sup_{\substack{(x,v),(y,\xi)\in supp[\rho^t]}}\Lip(\gen)_{(x-y,v-\xi)}^2 
% )\nonumber
%\end{eqnarray}
\begin{eqnarray}\nonumber
C(t)&:=&
\left(
4
\int_0^t
%%\|\gen\|_{{L^\infty(supp(\rho_1(s))}}^{2}
%\sup_{\substack{k,l=1,dots,N\\(Y,\Xi)\in supp[\rho^t_N]}}|\gen(y_k-y_l,\xi_k-\xi_l)|^2
%e^{\int_s^t\L (u)du}ds
%\int_0^t%\sup_{\substack{k,l=1,dots,N\\(Y,\Xi)\in supp[\rho^t_N]}}|\gen(y_k-y_l,\xi_k-\xi_l)|^2
\sup_{(x,v),(x',v')\in supp[\rho^t]}|\gen(x-x',v-v')|^2
e^{\int_s^tL (u)du}ds
\right)^{\frac12}\,
\nonumber\\
&\leq&
2\|\psi\|_\infty\left(
\int_0^t
%%\|\gen\|_{{L^\infty(supp(\rho_1(s))}}^{2}
%\sup_{\substack{k,l=1,dots,N\\(Y,\Xi)\in supp[\rho^t_N]}}|\gen(y_k-y_l,\xi_k-\xi_l)|^2
%e^{\int_s^t\L (u)du}ds
%\int_0^t%\sup_{\substack{k,l=1,dots,N\\(Y,\Xi)\in supp[\rho^t_N]}}|\gen(y_k-y_l,\xi_k-\xi_l)|^2
\sup_{(x,v),(x',v')\in supp[\rho^t]}|v-v'|^2
e^{\int_s^tL (u)du}ds
\right)^{\frac12}.\,
\nonumber\\
%\mbox{ and }\nonumber\\
% \L(t)&:=&2(1+2
% %\sup_{\substack{k,l=1,\dots,d\\(Y,\Xi)\in support[\Phi(t)\#(\rho^{in})^{\otimes N}]}}\Lip(\gen)_{(y_k-y_l,\xi_k-\xi_l)}^2
%\sup_{\substack{(x,v),(y,\xi)\in supp[\rho^t]}}\Lip(\gen)_{(x-y,v-\xi)}^2 
% ).\nonumber
L(t)&:=&
2(1+2
\min{(\sup_{\substack{i,l=1,\dots,N\\
(Y,\Xi)\in supp[\rho^t_N]}}\Lip{(\gen)}^2_{(y_i-y_l,\xi_i-\xi_l)},\sup_{\substack{
(x,v),(x',v')\in supp[\rho^t]}}\Lip{(\gen)}^2_{(x-x',v-v')})})
)\nonumber\\
&\leq&
2(1+2\sup_{\substack{i,l=1,\dots,N\\
(Y,\Xi)\in supp[\rho^t_N]}}\Lip{(\gen)}^2_{(y_i-y_l,\xi_i-\xi_l)})\nonumber\\
&\leq &
2(1+2\|\psi\|^2_\infty\sup_{\substack{i,l=1,\dots,N\\
(Y,\Xi)\in supp[\rho^t_N]}}
|\xi_i-\xi_l|^2).\nonumber
\end{eqnarray}
%We will need just a very little part of the stability results expressed by Ha, Kim and Zhuang, namely the following lemma, taken from  formula $(8)$ in Lemma 2.2 in \cite{hakimzhang}:
% $$
% \frac d{dt}\sup_{k,l=1,\dots,N}|v_k(t)-v_l(t)|\leq 0,\ \forall t,
% $$
% which leads naturally to
We will need just a very little part of the stability results expressed by Ha, Kim and Zhang in \cite{hakimzhang}, namely the following inequality:
 $$
 \frac d{dt}\sup
% _{k,l=1,\dots,N}
_{\substack{i,l=1,\dots,N\\(X,V)\in supp[\rho_N^{in}]}}|v_k(t)-v_l(t)|\leq 0,\ \forall t.
 $$
 Indeed,   formula $(8)$ in Lemma 2.2 in \cite{hakimzhang} stipulates that $\tfrac d{dt}\mathcal D_V(t)\leq 0$, where  $\mathcal D_V(t)$, as defined in Corollary $1$ of \cite{hakimzhang}, is precisely $ \sup\limits
 %_{k,l=1,\dots,N}
 _{\substack{i,l=1,\dots,N\\(X,V)\in supp[\rho_N^{in}]}}|v_k(t)-v_l(t)|$ in the case $p=2$ of the definition (4) of $\mathcal D_V(0):=\mathcal D_V$ in \cite{hakimzhang}.
 
 This inequality leads naturally to
 $$
 \sup_{\substack{i,l=1,\dots,N\\(X,V)\in supp[\rho_N^t]}}|v_i-v_l|^2
 \leq 4\|v\|_{L^\infty(supp[\rho^{in}])}^2,
 $$
 which implies
 $$
 L(t)\leq 2(1+8\|\psi\|^2_\infty\|v\|_{L^\infty(supp[\rho^{in}])}^2):=L.
 $$
%\be\label{lcs}
%\L(t):=
%%\sup_{\substack{k,l=1,\dots,d\\(Y,\Xi)\in support[\Phi(t)\#\rho_1^{\otimes N}]}}\Lip(\gen)_{(y_k-y_l,\xi_k-\xi_l)}^2
%2(1+2
% %\sup_{\substack{k,l=1,\dots,d\\(Y,\Xi)\in support[\Phi(t)\#(\rho^{in})^{\otimes N}]}}\Lip(\gen)_{(y_k-y_l,\xi_k-\xi_l)}^2
%\sup_{\substack{(x,v),(y,\xi)\in supp[\rho^t]}}\Lip(\gen)_{(x-y,v-\xi)}^2 
% )
%=
%2(1+2
% %\sup_{\substack{k,l=1,\dots,d\\(Y,\Xi)\in support[\Phi(t)\#(\rho^{in})^{\otimes N}]}}\Lip(\gen)_{(y_k-y_l,\xi_k-\xi_l)}^2
%\|\psi\|^2_\infty\sup_{\substack{(x,v),(y,\xi)\in supp[\rho^t]}}|v-\xi|^2 
% )
%\ee 
%and
%$$
%C(t):=
%%\left(\int_0^t\|\gen\|_{{L^\infty(supp(\rho_1(s))}}^{2}e^{\int_s^t\L (u)du}ds\right)^{\frac12}\,,
%\left(
%4
%\int_0^t
%%\|\gen\|_{{L^\infty(supp(\rho_1(s))}}^{2}
%\sup_{\substack{k,l=1,dots,N\\(Y,\Xi)\in supp[\rho^t_N]}}|\gen(y_k-y_l,\xi_k-\xi_l)|^2
%e^{\int_s^t\L (u)du}ds\right)^{\frac12}
% $$
% that is 
% \be\label{gcs}
%% \|\gen\|_{{L^\infty(supp(\rho_1(s))}}^{2}=\|\psi\|^2_\infty
%% \|v\|_{{L^\infty(supp(\rho_1(s)(x,v))}}^{2}.
%\sup_{\substack{k,l=1,dots,N\\(Y,\Xi)\in supp[\rho^t_N]}}|\gen(y_k-y_l,\xi_k-\xi_l)|^2
%=
%\|\psi\|^2_\infty
%\sup_{\substack{k,l=1,\dots,N\\(Y,\Xi)\in supp[\rho^t_N]}}|\xi_k-\xi_l|^2
%\ee
% 
% \vskip 1cm
%% We fist remark that, in the Cucker-Smale situation,
%% $$
%%  \Lip(\gen)_{(x,v)}\leq \Lip(\psi)|v|.
%% $$

We will  estimate $\sup\limits_{\substack{(x,v),(y,\xi)\in supp[\rho^t]}}|v-\xi|^2 \leq 2\sup\limits_{\substack{(x,v),(y,\xi)\in supp[\rho^t]}}(|v|^2+|\xi|^2)$ thanks to  Lemma 3.2 in \cite{prt} which stipulates that
 $$
 supp_v[\rho^t]\subset B(\bar v, V(t))
 $$
 with $\bar v=\int v\rho^{in}dxdv$ and $\frac d{dt}V(t)\leq 0$.
 
 Therefore $\|v\|_{L^\infty(supp[\rho^t])}\leq |\bar v|+V(0)$ so that, since one can take $V(0)=|\bar v|+|supp[\rho^{in}]|$,
 $$
 \sup\limits_{\substack{(x,v),(x'v')\in supp[\rho^t]}}|v-v'|^2
 \leq
 2\|\psi\|^2_\infty|(2|\bar v|+|supp[\rho^{in}]|)^2.
 $$ and 
\begin{eqnarray}
 C(t)
 &\leq&
  2\|\psi\|^2_\infty\left(
2\int_0^t
(2|\bar v|+|supp[\rho^{in}]|)^2
e^{L(t-s)}ds
\right)^{\frac12}\nonumber\\
&=&
2\|\psi\|^2_\infty\left(
2
(2|\bar v|+|supp[\rho^{in}]|)^2
\frac{e^{Lt}-1}L
\right)^{\frac12}.\nonumber
 \end{eqnarray}
% Moreover,
% $$
% \sup_{\substack{k,l=1,\dots,N\\(Y,\Xi)\in supp[\rho_N^t]}}|\xi_k-\xi_l|^2
% =
% \sup_{\substack{k,l=1,\dots,N\\(X(0),V(0))\in supp[(\rho^{in})^{\otimes N}]}}|v_k(t)-v_l(t)|^2
% $$
% where $x_i(t0,v_i(t),\ i=1,\dots, N$ satisfies \eqref{CS}.
 
% We will need just a very little part of the stability results expressed by Ha, Kim and Zhuang, namely the following lemma, taken from  formula $(8)$ in Lemma 2.1 in \cite{hakimzhang}:
% $$
% \frac d{dt}\sup_{k,l=1,\dots,N}|v_k(t)-v_l(t)|\leq 0,\ \forall t,
% $$
% which leads naturaly to
% $$
% \sup_{\substack{k,l=1,\dots,N\\(Y,\Xi)\in supp[\rho_N^t]}}|\xi_k-\xi_l|^2
% \leq 4\|v\|_{L^\infty(\rho^{in})}^2
% $$
%% For the term $\|v\|_{L^\infty(supp[\rho_1(t)]}$ we will use this time Lemma 3.2 in \cite{prt} which stipulate that
%% $$
%% supp_v[\rho_1(t)\subset B(\bar v, V(t))
%% $$
%% with $\bar v=\int v\rho^{in}dxdv$ and $\frac d{dt}V(t)\leq 0$.
%% 
%% Therefore $\|v\|_{L^\infty(supp[\rho^t]}\leq |\bar v|+V(0)$ so that, since one can take $V(0)=|\bar v|+|supp[\rho^{in}|$,
%% $$
%% \|\gen\|_{{L^\infty(supp(\rho_1(s))}}^{2}
%% \leq
%% \|\psi\|^2_\infty|(2|\bar v|+|supp[\rho_1(0)|)^2.
%% $$
% which implies
% $$
% C(t)\leq 4\left(
% \|\psi\|^2_\infty||v\|_{L^\infty(\rho^{in})}^2\frac{e^{Lt}-1}L\right)^{\frac12}.
% $$
 
 Theorem \ref{main} is proved.

\begin{appendix}
\section{Dynamical estimates for general Cucker-Smale particle systems}\label{dynestpar}
In this section, we give global estimates on the flow $\Phi^N(t)$ generated by \eqref{genCS}.

We have, for each $i=1,\dots,N$,
\begin{eqnarray}
\frac d{dt}|v_i|&\leq&
|\dot v_i|\nonumber\\
&\leq&
\frac1N \sum_{j=1}^N
|\gamma(x_i-v_j,v_i-v_j)|\leq
\frac1N \sum_{j=1}^N
\gen_0|v_i-v_j|\nonumber\\
&\leq&
\frac1N \sum_{j=1}^N\gen_0(|v_j|+|v_i|),\label{crossfingers}
\end{eqnarray}
so that
$$
\frac d{dt}\sum_{i=1}^N|v_i|
\leq2\gen_0 \sum_{i=1}^N|v_i|
$$
and, by Gronwall inequality,
\be\label{crucial}
\sum_{i=1}^N|v_i(t)|\leq \sum_{i=1}^N|v_i(0)|e^{2\gen_0t}.
\ee
Turning back to \eqref{crossfingers}, we get by \eqref{crucial}, for each $i=1,\dots,N$,
\begin{eqnarray}
\frac d{dt}(|v_i|)&\leq& \gen_0|v_i|+\frac{\gen_0}N\sum_{j=1}^N|v_j|\leq\gen_0|v_i|+\gen_0e^{2\gen_0t}\frac1N\sum_{j=1}^N|v_j(0)|\nonumber\\
&\leq& 
\gen_0|v_i|+\gen_0e^{2\gen_0t}\max_{j=1,\dots,N}|v_j(0)|.\nonumber
\end{eqnarray}
Therefore, again by Gronwall Lemma and uniformly in $N$,
\begin{eqnarray}\label{estpart}
|v_i(t)|&\leq& e^{\gen_0t}\big(
|v_i(0)+\gen_0\max_{j=1,\dots,N}|v_j(0)|\frac{e^{\gen_0t}-1}{\gen_0}\big)\nonumber\\
&\leq& \max_{j=1,\dots,N}|v_j(0)|e^{2\gen_0t},\ \ \ i=1,\dots,N.
\end{eqnarray}
Finally, \eqref{genCS} gives immediatly that
$$
||x_1(t)|-|x_i(0)||
\leq 
\max_{j=1,\dots,N}|v_j(0)|\frac{e^{2\gen_0t}-1}{2\gen_0},\ \ \ i=1,\dots,N.
$$

\section{Dynamical estimates for general Cucker-Smale kinetic systems}\label{dynestpkin}
Let us recall that, according to Theorem 2.3 in \cite{prt}, there exists a diffeomorphism $\Phi(t)$ on  $\bR^{2d}$ such that the solution $\rho^t$ of \eqref{vlasovgencs} is given by
$$
\rho(t)=\Phi(t)\#\rho^{in}.
$$
Moreover, $\Phi(t)$ solves the system
$$
\dot\Phi(t)
:=
\binom{\dot\Phi_x(t)}{\dot\Phi_v(t)}=\binom{\Phi_v(t)}{\big(\gen *\rho^t\big)(\Phi(t)):=\int \gen(\Phi_x(t)-y,\Phi_v(t)-\xi)\rho^t(y,\xi)dxd\xi}.
$$
Therefore,
\be\label{avantlemme}\frac d{dt} |\Phi_v(t)|
\leq |\dot\Phi_v(t)|
\leq
\gen_0(\|\rho^t
\|_1|\Phi_v(t)|+\int|\xi|\rho^t(y,\xi)dyd\xi).
\ee
Obviously $\|\rho^t
\|_{L^1}=\|\rho^{in}
\|_{L^1}=1$.
\begin{Lem}\label{crucialvlasov}
$$
\int|
v|
 \rho^tdxdv\leq e^{2\gen_0t}\int|
v|
 \rho^{in}dxdv.
 $$
\end{Lem}
\begin{proof}
\begin{eqnarray}
\frac d{dt}\int |v|\rho^tdxdv
&
\leq&|
\int |v|
\dot \rho^tdxdv|\nonumber\\
&\leq&
|\int |v|
\nabla_v(\int\gen(x-y.v-\xi)\rho^t(y,\xi)\rho^t(x,v)dxdvdyd\xi|\nonumber\\
&\leq&
\int|\frac v{|v|} \gen_0((v|+|\xi|)\rho^t(y,\xi)\rho^t(x,v)dxdvdyd\xi%\nonumber\\
%&
\leq
%&
2\gen_0\int|
v|
 \rho^tdxdv.\nonumber
\end{eqnarray}
The result follows by Gronwall inequality.
\end{proof}
Thanks to Lemma \ref{crucialvlasov}, \eqref{avantlemme} becomes
$$
\frac d{dt} |\Phi_v(t)|
\leq
\gen_0|\Phi_v(t)|+\gen_0e^{2\gen_0t}\int|v|\rho^{in}dxdv,
$$
and, by Gronwall Lemma, we have
\begin{eqnarray}\label{estiphi}
|\Phi_v(t)(x,v)|
&\leq&
e^{\gen_0t}(|v|
+
(e^{\gen_0t}-1)\int|v|\rho^{in}dxdv)\nonumber\\
&\leq &
e^{2\gen_0t}(\|v\|_{L^\infty(supp[\rho^{in}])}+
\|v\|_{L^1(supp[\rho^{in}])})
.
\end{eqnarray}

\end{appendix}
\vskip 1cm
\textbf{Acknowledgements.} The work of Thierry Paul was partly supported by LIA LYSM (co-funded by AMU, CNRS, ECM and INdAM) and IAC (Istituto per le Applicazioni del Calcolo "Mauro Picone") from CNR.

\end{document}